\documentclass[10pt]{article} 

\usepackage{amsmath,amsthm,amsfonts,amssymb,amscd,latexsym,eucal}
\usepackage[a4paper,top=3cm,bottom=4.5cm,left=2.5cm,right=2.5cm]{geometry}

\newtheorem{corollary}{Corollary}
\newtheorem{definition}{Definition}

\newtheorem{proposition}{Proposition}
\newtheorem{remark}{Remark}

\author{Carlo Pandiscia}
\title{Factorization of stochastic maps using the Stinespring representations}

\begin{document}

\maketitle

\begin {abstract}
In this work, we investigate the existence of a factorization  for a unital completely positive map, between non-commutative probability space  which do not change the expectation values of the events. These  maps  are called in literature stochastic maps. Using the Stinespring representations of completely positive map  and assuming the existence of anti-unitary operator on Hilbert space related to these representations which satisfying some modular relations, we prove that stochastic maps with adjoint, admits a factorization. 
\end {abstract}

\section{Introduction}

We study the existence of a factorization for a preserving Markov operators,
question posed and debated by Anantharaman-Delaroche in \cite{Antha}. The factorization problem is closely connected with the existence of a reversible dilation of a quantum dynamical system in the sense of Kummerer \cite{Kummerer}, since Haagerup and Musat they proved in \cite{Haag-Musat} that the factorization property is equivalent to the existence of a reversible dilation of quantum dynamical system.
\\
The preserving Markov operator between commutative probability spaces are factorizable and each deterministic map \textit{i.e.} multiplicative preserving Markov operator between generic probability space have this property since they admits a reversible dilation (see e.g. \cite{Haag-Musat}, \cite{Kummerer}, \cite{panda2}). 
\\
In this paper, we give a constructive methods for to determine a factorization of preserving Markov map using the Stinespring representations and assuming the existence of anti-unitary operator $\hat{J}$ on Hilbert space of the Stinespring representation. In commutative and deterministic case we have a natural choice for  the anti-unitary operator $\widehat J $ which happens to be a conjugation.
\\
We want to underline that the problem of to establish when a map admits a factorization and hence a reversible dilation, is a problem that still remains largely open,  although it was posed many years ago \cite{Kummerer}.
\\
this paper is organized as follows:
\\
In section 2 we recall the main definitions and results on completely positive maps and quantum dynamical systems that can be found in \cite{Brown-Ozawa}, \cite{NSZ} and \cite{Paulsen}. 
\\
In section 3 Using the Stinespring representaions related to preserving Markov operator $\Phi$ and assuming the  existence of particular anti-unitary operator on Hilbert space of Stinespring representation, we prove that  $\Phi$ admits a factorization.
\\
In section 4, After that we have introduced briefly  the notion of the generalized conditional expectation of Accardi and Cecchini, which the reader can found in  \cite{Accardi-Cecchini} and \cite{Cecchini},  we prove that if this map is a Umegaki conditional expectation \cite{Umegaki}, \cite{Cuculescu}, then the preserving Markov operator admits a factorization.
  
\section{Preliminaries}
In this paper we consider the  probability spaces $(\mathfrak{M},\varphi)$ constituted by a von Neumann algebra  $\mathfrak{M}$ and by its normal faithful state $\varphi$. 
\\
Let $(\mathfrak{M},\varphi)$ be a probability space, we set with $( \mathcal{H}_{\varphi},\pi_{\varphi}, \Omega_{\varphi})$ the GNS representation of the normal state $\varphi$ and with $(J_{\varphi},\Delta_{\varphi})$ the modular operators associated with the von Neumann algebra $\pi_{\varphi}(\mathfrak{M})$ and with $\left\{\sigma_{t}^{\varphi} \right\}$ its modular group.
\\
Furthermore we set with $\mathfrak B(\mathcal H)$ the von Neumann algebra of bounded operators on Hilbert space $\mathcal H$.
\\
A \textbf{stochastic map} $\Phi:(\mathfrak M_1  ,\varphi_1)\rightarrow (\mathfrak M_2 ,\varphi_2)$ between probability space $\left\{\mathfrak M_i ,\varphi_i\right\}$  with $i=1,2$, is a normal unital completely positive map $\Phi:\mathfrak M_1\rightarrow\mathfrak M_2$ with the following property $\varphi_2\circ \Phi=\varphi_1$. 
\\
We have a normal unital completely positive map $\Phi_{\bullet}:\pi_{\varphi_1}(\mathfrak M_1)\rightarrow \pi_{\varphi_2}(\mathfrak M_2)$ such that
\[
\Phi_{\bullet}(\pi_{\varphi}(A))=\pi_{\omega}(\Phi(A))
\]
for all $A\in\mathfrak M_1$.
\\
Moreover, there is a linear contraction $U_{\Phi}:\mathcal{H}_{\varphi_1}\rightarrow \mathcal{H}_{\varphi_2}$ defined as
\begin{equation}
\label{lincontr}
U_{\Phi}\pi_{\varphi_1}(A) \Omega_{\varphi_1}=\pi_{\varphi_2}(\Phi(A)) \Omega_{\varphi_2} 
\end{equation}
for all $A\in\mathfrak M_1$.
\\
Furthermore the  Stochastic map $\Phi$ admits a $(\varphi_1,\varphi_2)-$adjoint if there is a stochastic map $\Phi^{\sharp}:(\mathfrak M_2 , \varphi_2)\rightarrow(\mathfrak M_1 , \varphi_1)$ such that
\[
\varphi_2 (B \ \Phi(A))=\varphi_1(\Phi^{\sharp}(B) \ A))
\]
for all $A\in\mathfrak M_1$ and $ B\in\mathfrak M_2$.
\\
A stochastic map $\Phi$ between two probability spaces is said be a \textbf{deterministic map} whether is  a homomorphism of von Neumann algebras.
\\
We have a fundamental proposition (see \cite{Accardi-Cecchini} and \cite{NSZ}):
\begin{proposition}
Let $\Phi:(\mathfrak M_1 , \varphi_1 )\rightarrow (\mathfrak M_2 , \varphi_2)$ be a stochastic map, the following conditions are equivalent:
\begin{itemize}
\item $\Phi$ admits $(\varphi_1,\varphi_2)-$adjoint
\item $\Phi_{\bullet}\circ\sigma_{t}^{\varphi_1}=\sigma_{t}^{\varphi_2}\circ\Phi_{\bullet}  \qquad  t\in\mathbb{R}$ 
\item $J_{\varphi_2} U_{\Phi}= U_{\Phi}J_{\varphi_1}$.
\end{itemize}
\end{proposition}
If the equivalent conditions of the previous proposition are satisfied, then we say that $\Phi$ is a $(\varphi_1 ,\varphi_2)$-\textbf{preserving Markov map} \cite{Antha}.
\\
We have the following definition:
\begin{definition}
Let $\Phi:(\mathfrak M_1 ,\varphi_1)\rightarrow(\mathfrak M_2 , \varphi_2)$ be a $(\varphi_1,\varphi_2)$-preserving Markov map. We say that $\Phi$ is a \textbf{factorizable map} if there exists a probability space $\left\{ \mathfrak R , \omega \right\}$ and two deterministic preserving Markov operators $\alpha:(\mathfrak M_2 , \varphi_2)\rightarrow(\mathfrak R ,\omega)$ and  $\beta:(\mathfrak M_1 ,\varphi_1)\rightarrow (\mathfrak R , \omega )$ such that $\Phi=\alpha^{\sharp}\circ\beta$.
\\
The factorization $(\alpha , \beta)$ is minimal if 
$$ \mathfrak R = \alpha(\mathfrak M_2) \vee \beta(\mathfrak M_1)$$
where $\alpha(\mathfrak M_2) \vee \beta(\mathfrak M_1)$ is the von Neumann algebra generated by $\alpha(\mathfrak M_2)$ and $\beta(\mathfrak M_1)$.
\end{definition}
We underline that a factorization $(\alpha , \beta)$ of the preserving Markov operator $\Phi$ determine a factorization of linear contraction $U_\Phi$ since
$$U_\Phi= U_\alpha ^*  \ U_\beta $$
where $U_\alpha$ and $U_\beta$ are the linear isometries defined in (\ref{lincontr}) related to $\alpha$ and $\beta$ homomorphism.
\\

In the following of the discussion, unless noted otherwise,  we will consider the probability spaces $( \mathfrak M , \Omega)$ in standard form \textit{i.e.} concrete von Neumann algebra $\mathfrak M \subset \mathfrak B(\mathcal H)$ with $\Omega$ cyclic and separating vector of Hilbert space $\mathcal H$ and stochastic map $\Phi:(\mathfrak M_1  ,\Omega_1)\rightarrow (\mathfrak M_2 ,\Omega_2)$ such that
$$   \left\langle   \Omega_2, \Phi(A ) \Omega_2 \right\rangle =  \left\langle  \Omega_1, A  \Omega_1 \right\rangle $$
for all $ A\in\mathfrak M_1 $.
\\
For a preserving Markov operator $\Phi:(\mathfrak M_1  ,\Omega_1)\rightarrow (\mathfrak M_2 ,\Omega_2)$,  its dual  
map \cite{mohari} $\Phi':(\mathfrak M'_2  ,\Omega_2)\rightarrow (\mathfrak M'_1 ,\Omega_1)$ such that
$$\left\langle A \Omega_1  |  \Phi'(Y) \Omega_1  \right\rangle= \left\langle \Phi(A) \Omega_2  |  Y \Omega_2  \right\rangle$$
for all $A\in\mathfrak M_1$ and $Y\in\mathfrak M_2'$ is defined as 
$$ \Phi'(Y) = J _1 \Phi^{\sharp}(J_2 Y J_2 ) J_1 $$
We observe that if $(\mathfrak M, \varphi)$ is a probability space then $\varphi$ is  factorizable.
\\
Indeed, we can define two preserving Markov maps $\alpha,\beta: (\mathfrak M,\varphi) \rightarrow (\pi_\varphi(\mathfrak M) \overline{\otimes} \pi_\varphi (\mathfrak M), \Omega_\varphi \otimes \Omega_\varphi)$ where for any $a\in\mathfrak M$:
$$ \alpha(a)=\pi_\varphi(a) \otimes 1 \qquad \text{and} \qquad \beta(a)=1 \otimes \pi_\varphi(a)$$
and $\pi_\varphi(\mathfrak M) \overline{\otimes} \pi_\varphi (\mathfrak M)$ is the von Neumann algebra of $\mathfrak{B}(\mathcal{H}_\varphi\otimes\mathcal{H}_\varphi)$ weakly closure of the *-algebra generated by elements $\overset{n}{\underset{i}{\sum}}A_i\otimes B_i$, with $A_i, B_i\in\pi_\varphi(\mathfrak M)$.
\\
Furthermore for the adjoint maps we have:
$$\alpha^\sharp(A\otimes B)= \left\langle \Omega_\varphi , B\Omega_\varphi \right\rangle A \qquad \text{and} \qquad \beta^\sharp(A\otimes B)= \left\langle \Omega_\varphi , A\Omega_\varphi \right\rangle B$$
for all $A,B\in\pi_\varphi (\mathfrak M)$ 
and
$\varphi(a) I = \beta^\sharp (\alpha(a))=\alpha^\sharp (\beta(a))$ for all $a\in\mathfrak M$. 
\\

We recall briefly the Stinespring representations associated to unital completely positive maps \cite{Paulsen}.
\\
Let $\mathfrak M_1$ and $\mathfrak M_2 $ be a concrete von Neumann algebra of $\mathcal{B}(\mathcal H_1)$ and $\mathcal{B}(\mathcal H_2)$ respectively.
 \\
Let $\Phi:\mathfrak M_1 \rightarrow\mathfrak M_2 $ be a normal unital completely positive map  
\\
On the algebraic tensor  $\mathfrak M_1 \otimes\mathcal H_2$ we can define a semi-inner product by
\[
\left\langle A \otimes h| X \otimes k \right\rangle=
\left\langle h ,\Phi\left( A ^{\ast} X\right) k \right\rangle
\]
for all $A,X\in\mathfrak{M}_1$ and $h, k\in\mathcal H_2 $.
\\
Furthermore the Hilbert space $\mathcal{L}_{\Phi}$ is the completion of the quotient space 
$\mathfrak M_1 \overline{\otimes}_{\Phi}\mathcal H_2 $ of $\mathfrak M_1 \otimes\mathcal H_2$ by the linear subspace
\begin{equation}
\label{V-space}
\mathcal{V}=\left\{ l \in\mathfrak{M}_1\otimes\mathcal{H}_2 :\left\langle l \ |  \ l \right\rangle = 0 \right\}
\end{equation}
with inner product induced by $\left\langle \cdot \ ,\cdot\right\rangle$. 
\newline
We shall denote the image at  $A\otimes h\in\mathfrak{M}_1\otimes\mathcal{H}_2$ in $\mathfrak{M}_1\overline{\otimes}_{\Phi}\mathcal{H}_2$ 
by $A\overline{\otimes}_{\Phi}h $, so that we have
\[
\left\langle A \overline{\otimes}_{\Phi}\ h , X \overline{\otimes}_{\Phi} k \right\rangle _{\mathcal{L}_{\Phi}}=
\left\langle h ,\Phi\left(A^{\ast} X \right) k \right\rangle
\]
for all $A, X \in\mathfrak{M}_1$ and $ h , k \in\mathcal{H}_2$.
\\
Moreover we can define a representation $\sigma_{\Phi}:\mathfrak{M}_1\rightarrow\mathcal{B}(\mathcal{L}_{\Phi})$ defined by 
\[
\sigma_{\Phi}(A)(X \overline{\otimes}_{\Phi} h) = AX \overline{\otimes}_{\Phi} h,
\]
for each $A\overline{\otimes}_{\Phi} h\in\mathcal{L}_{\Phi}$ and $V_{\Phi} h =\mathbf{1}\overline{\otimes}_{\Phi}h$ 
for each $h\in\mathcal{H}_2$
\\
Since $\Phi$ is a unital map, the linear operator $V_{\Phi}$ is an isometry with adjoint $V_{\Phi}^{\ast}$ 
defined as
\[
V_{\Phi}^* A\overline{\otimes}_{\Phi} h = \Phi(A) h 
\] 
for all $A\in\mathfrak M_1 $ and $h\in\mathcal{H}_2$.
\\
We can define the following linear operator $\Lambda_{\Phi}:\mathcal H_1 \rightarrow  \mathcal{L}_{\Phi}$:
\[
\Lambda_{\Phi} A\Omega_1 = A \overline{\otimes}_{\Phi}\Omega_2
\]
for all $A\in\mathfrak{M}_1$
\\
We remark that 
$$U_\Phi= V^*_{\Phi} \Lambda_{\Phi}$$
It is easy to prove that for each $A\in\mathfrak{M}_1$ and $h\in\mathcal{H}_2$ we have:
\[
\Lambda^*_{\Phi}A\overline{\otimes}_{\Phi}h = A U_{\Phi}^*h
\]
Furthermore
\begin{equation}
\label{st1}
\Lambda^*_{\Phi}\sigma_{\Phi}(A)\Lambda_{\Phi}=A,
\end{equation}
and
\begin{equation}
\label{st2}
\Lambda_{\Phi}\Lambda_{\Phi}^*\in\sigma_{\Phi}(\mathfrak{M}_1)'
\end{equation}
We have a new statement:
\begin{proposition}
There is a normal representation $\tau_{\Phi}:\mathfrak{M}_2'\rightarrow\mathcal{B}(\mathcal{L}_{\Phi})$ such that for each $Y\in\mathfrak{M}'_2$ and $A\overline{\otimes}_{\Phi}h\in\mathcal{L}_{\Phi}$
\[
\tau_{\Phi}(Y) A\overline{\otimes}_{\Phi}h = A\overline{\otimes}_{\Phi}Yh,
\]
with
\begin{equation}
V_{\Phi}^*\tau_{\Phi}(Y) V_{\Phi}=Y
\end{equation}
Furthermore 
\begin{equation}
\label{stdual}
\tau_{\Phi}(\mathfrak{M}'_2)\subset\sigma_{\Phi}(\mathfrak{M}_1)',
\end{equation}
and
\begin{equation}
V_{\Phi} V_{\Phi}^* \in\tau_{\Phi}(\mathfrak{M}'_2)'
\end{equation}
\end{proposition}
\begin{proof}
We fix a vector $ l =\overset{n}{\underset{i}{\sum}} A_{i}\otimes h_{i}\in\mathfrak{M}_1\otimes\mathcal{H}_2 $ and we get the following linear functional on 
$\mathfrak{M}'_2$
\[
\omega_{l}(Y)   = \left\langle l | Y l \right\rangle = \overset{n}{\underset{i,j}{\sum}}\left\langle h_{i},\Phi(A_{i}^{\ast}A_{j})Y h_{j}\right\rangle
\]
for all $Y\in\mathfrak{M}'_2$.
\\
The linear functional $\omega_{l}$ is positive since 
\[
\omega_{l}(Y^*Y)=\overset{n}{\underset{i,j}{\sum}}\left\langle Y h_{i},\Phi(A_{i}^{\ast}A_{j})Y h_{j}\right\rangle\geq0
\]
then $\omega_{l}$ is a continuous functional \cite{Bra-Rob}  with $\omega_{l}(1) = \left\langle l | l \right\rangle$  and  
$$|\omega_{l}(Y^*Y)|\leq||Y||^2\omega_{l}(1)=||Y||^2\left\langle l | l \right\rangle$$ 
For each $Y\in\mathfrak{M}'_2$ we can define
\[
\tau_{o}(Y) \overset{n}{\underset{i}{\sum}} A_{i}\otimes h_{i}=\overset{n}{\underset{i}{\sum}} A_{i}\otimes Yh_{i}.
\]
We observe that if $l=0$ then $\tau_{o}(Y)l=0$ since $||\tau_{o}(Y)l||^{2}=|\omega_{l}(Y^*Y)|\leq||Y||^2\left\langle l | l \right\rangle=0$.
\newline
Therefore $\tau_{o}(\mathfrak{M}')\mathcal{V}\subset\mathcal{V}$ where $\mathcal{V}$ is the linear space (\ref{V-space}). It follows that $\tau_{\Phi}:\mathfrak M'_2 \rightarrow\mathcal{B}(\mathcal{H})$ defined as 
\[
\tau_{\Phi}(Y) \overset{n}{\underset{i}{\sum}} A_{i}\overline{\otimes}_{\Phi} h_{i}= \overset{n}{\underset{i}{\sum}} A_{i}\overline{\otimes}_{\Phi} Yh_{i},
\]
for all $Y \in\mathfrak{M}'_2$ is a well-defined representation of the von Neumann algebra $\mathfrak{M}'_2$.
\\
Let $\left\{Y_{\alpha}\right\}_{\alpha}$ be a net in $\mathfrak{M}'$ such that $Y_{\alpha}\rightarrow Y$ in $\sigma$-top,
for each $A\overline{\otimes}_{\Phi}h\in\mathcal{L}_{\Phi}$ we obtain:
\[
\left\langle A\overline{\otimes}_{\Phi}h, \tau_{\Phi}(Y_{\alpha}) A\overline{\otimes}_{\Phi}h\right\rangle=
\left\langle h,\Phi(A^*A)Y_{\alpha}h \right\rangle\rightarrow
\left\langle h, \Phi(A^*A)Y h\right\rangle=
\left\langle A\overline{\otimes}_{\Phi}h, \tau_{\Phi}(Y) A\overline{\otimes}_{\Phi}h\right\rangle
\]
therefore the representation $\tau_{\Phi}$ is $\sigma$-top continuous.
\\
The others relationships are straightforward.
\end{proof}
If $\Phi:(\mathfrak M_1  ,\Omega_1)\rightarrow (\mathfrak M_2 ,\Omega_2)$ is a preserving Markov operator then 
\begin{equation}
\label{st5}
\Lambda_{\Phi}^*\tau_{\Phi}(Y)\Lambda_{\Phi}= \Phi'(Y)
\end{equation}
for all $Y\in\mathfrak M_2'$.
\\
Indeed for each $A\in\mathfrak{M}_1$ we have
\begin{eqnarray*}
\Lambda_{\Phi}^*\tau_{\Phi}(Y)\Lambda_{\Phi}A\Omega_1 &=&
\Lambda_{\Phi}^*\tau_{\Phi}(Y) A\overline{\otimes}_{\Phi}\Omega_2 =
\Lambda_{\Phi}^*A\overline{\otimes}_{\Phi}Y\Omega_2= A U_{\Phi}^{*}Y\Omega_2
\\
&=&
A J_1 U_{\Phi}^{*}J_2 Y J_2 \Omega_2 = A J _1 \Phi^{\sharp}(J_2 Y J_2 )\Omega_1=
\\
&=&
J_1 \Phi^{\sharp}(J_2 Y J_2) J_1 A J_1 \Omega_1=
J_1 \Phi^{\sharp}(J_2 Y J_2)J_1 A\Omega_1 =
\\
&=& \Phi'(Y) A\Omega_1
\end{eqnarray*}
We observe that $\sigma_{\Phi}(\mathfrak{M}_1)$ and $\tau_{\Phi}(\mathfrak{M}_2')$  are von Neuman algebras, since the representations $\sigma_{\Phi}:\mathfrak{M}_1\rightarrow\mathcal{B}(\mathcal{L}_{\Phi})$ and $\tau_{\Phi}:\mathfrak{M}_2'\rightarrow\mathcal{B}(\mathcal{L}_{\Phi'})$ are normal maps.
\\
Furthermore, for each $A\in\mathfrak M_2$ and $Y\in\mathfrak M'_1$ we have:
$$ V_{\Phi}^* \sigma_{\Phi}(A)\tau_{\Phi}(Y)V_{\Phi}=\Phi(A) Y \in\mathfrak M_2 \cdot \mathfrak M'_2 \subset \mathfrak B(\mathcal H_2)$$
while 
$$\Lambda_{\Phi}^*\sigma_{\Phi}(A)\tau_{\Phi}(Y)\Lambda_{\Phi}= A \Phi'(Y) \in\mathfrak M_1 \cdot \mathfrak M'_1 \subset \mathfrak B(\mathcal H_1)$$
It follows that if $\mathfrak B = \sigma(\mathfrak M_1) \vee \tau_\Phi(\mathfrak M_2')$ is von Neumann algebra of $\mathfrak B(\mathcal L_\Phi)$ generated by $\sigma(\mathfrak M_1) $ and $\tau_\Phi(\mathfrak M_2')$, then we can define two unital completely positive maps $E_i:\mathfrak B \rightarrow \mathfrak M_i \cdot \mathfrak M_i ' $ with $i=1,2$ such that for any $T\in\mathfrak B$
$$E_1 (T) = V_{\Phi}^* T V_{\Phi} \qquad \text{and} \qquad  E_2 (T) = \Lambda_{\Phi}^* T \Lambda_{\Phi} $$
with
$$ \left\langle \Omega_i,E_i(T) \Omega_i \right\rangle = \left\langle \Omega_\Phi, T \Omega_\Phi \right\rangle $$ 
\\
Furthermore, for any $X_i\in \mathfrak M_2$ and $h_i\in \mathcal H_1$ with $i=1,2\ldots n$, we can define 
\begin{equation}
\label{antiuniw}
W \sum_{i=1}^{n } X_i\overline{\otimes}_{\Phi^\sharp} h_i   = \sum_{i=1}^{n}J_2 X_i J_2 \overline{\otimes}_{\Phi'} J_1 h_i
\end{equation}
where
\begin{align*}  
|| W \sum_{i=1}^{n } X_i\overline{\otimes}_{\Phi^\sharp} h_i  ||^2 & = 
\sum_{i,j}  \left\langle  J_2 X_i J_2 \overline{\otimes}_{\Phi ' } J_1 h_i , J_2 X_j J_2 \overline{\otimes}_{\Phi '} J_1 h_j \right\rangle =
\\
&=
\sum_{i,j}  \left\langle J_1 h_i ,\Phi ' (J_2 X^*_i X_j J_2 ) J_1 h_j \right\rangle = 
\\
&= 
\sum_{i,j}  \left\langle  h_i ,\Phi ^\sharp (X^*_i X_j  ) h_j \right\rangle =
|| \sum_{i=1}^{n } X_i\overline{\otimes}_{\Phi^\sharp} h_i  ||^2
\end{align*} 
In other words, we have an anti-unitary operator $W:\mathcal L_{\Phi^\sharp } \rightarrow \mathcal L_{\Phi'}$ such that
\begin{equation*}
W^* \tau_{\Phi'}(A_1) W = \tau_{\Phi^\sharp} (J_1 A_1 J_1) \quad \text{and} \quad W^* \sigma_{\Phi'}(Y_2) W = \sigma_{\Phi^\sharp} (J_2 Y_2 J_2)  
\end{equation*}
for all $A_1\in\mathfrak M_1$ and $Y_2\in\mathfrak M_2'$.
\section{Factorization of deterministic map}
In this section we prove that each deterministic preserving Markov operator $\Phi:(\mathfrak M_1, \Omega_1 )\rightarrow(\mathfrak M_2,\Omega_2)$ is factorizable.
\\
We set with $ (\mathcal L_{\Phi^\sharp}, \sigma_{\Phi^\sharp}, V_{\Phi^\sharp})$  and  $(\mathcal L_{\Phi^\sharp},\tau_{\Phi^\sharp}, \Lambda_{\Phi^\sharp})$ the Stinespring representatons of $\Phi^{\sharp}$, the adjoint of $\Phi$.
\\
We observe that the vector $\Omega_{\Phi^\sharp} = 1\overline{\otimes}_{\Phi^{\sharp}}\Omega_1$ is separable for the von Neumann algebra 
$\sigma_{\Phi^\sharp}(\mathfrak M_2)$ and  we have a probability space $(\sigma_{\Phi^\sharp}(\mathfrak M_2),\omega)$ where for any $T\in\sigma_{\Phi^\sharp}(\mathfrak M_2)$ we have defined  $\omega(T)=\left\langle \Omega_{\Phi^\sharp} ,T\Omega_{\Phi^\sharp} \right\rangle$.
\newline
Moreover the map $\sigma_{\Phi^\sharp}:(\mathfrak M_2,\Omega_2)\rightarrow(\sigma_{\Phi^\sharp}(\mathfrak M_2),\omega)$ is a preserving Markov operator with adjoint $\sigma_{{\Phi}^{\sharp}}^{\sharp}(T)=\Lambda^{\sharp *}T\Lambda^{\sharp}$ for all $T\in\sigma_{{\Phi}^{\sharp}}(\mathfrak M_2)$ since for each $A,B\in\mathfrak M_2$ we have
\[
\left\langle \Omega_{\Phi^\sharp} , \sigma_{{\Phi}^{\sharp}}(B) \sigma_{{\Phi}^{\sharp}}(A)\Omega_{\Phi^\sharp} \right\rangle=
\left\langle \Omega_1,\Phi^{\sharp}(BA)\Omega_1\right\rangle=
\left\langle \Omega_2, BA \Omega_2\right\rangle=
\left\langle \Omega_2,\Lambda_{\Phi^\sharp}^*\sigma_{\Phi^\sharp}(B)\Lambda_{\Phi^\sharp} A\Omega_2\right\rangle
\]
The map $\Theta:(\mathfrak M_1,\Omega_1)\rightarrow (\sigma_{\Phi^\sharp}(\mathfrak M_2),\omega)$ defined by
$$
\Theta(A)=\sigma_{{\Phi}^{\sharp}}(\Phi(A))
$$
for all $A\in\mathfrak M_1$, is a preserving Markov operator with adjoint $\Theta^{\sharp}(T)=V_{\Phi^\sharp}^*T V_{\Phi^\sharp}$ for all $T\in\sigma_{\Phi^\sharp}(\mathfrak M_2)$.
\\
Indeed for each $A\in\mathfrak M_1$ and $B\in\mathfrak M_2$ we obtain:
\begin{eqnarray*}
\left\langle \Omega_{\Phi^\sharp} , \sigma_{\Phi^\sharp}(B) \Theta(A)\Omega_{\Phi^\sharp} \right\rangle &=&
\left\langle \Omega_{\Phi^\sharp}  , \sigma_{\Phi^\sharp}(B\Phi(A))\Omega_{\Phi^\sharp} \right\rangle=
\left\langle \Omega_1, \Phi^{\sharp}(B\Phi(A)) \Omega_1\right\rangle=
\\
&=&
\left\langle \Omega_2,B\Phi(A)\Omega_2\right\rangle=
\left\langle \Omega_1,\Phi^{\sharp}(B) A \Omega_1 \right\rangle=
\\
&=&
\left\langle \Omega_1,\Theta^{\sharp}(\sigma_{{\Phi}^{\sharp}}(B) )A\Omega_1\right\rangle
\end{eqnarray*}
We have the following proposition:
\begin{proposition}
Any deterministic preserving Markov operator $\Phi:(\mathfrak M_1, \Omega_1 )\rightarrow(\mathfrak M_2,\Omega_2)$ admits a factorization. 
\end{proposition}
\begin{proof}
We have that $\Phi^{\sharp}(A)=\Theta^{\sharp}({\sigma}_{{\Phi}^{\sharp}}(A))$ for all $A\in\mathfrak M_1$ it follows that $\Phi=\sigma_{\Phi^\sharp}^{\sharp}\circ\Theta$.
\end{proof}

\section{Factorization and Stinespring representations}
We want to study the possibility of extending to any preserving Markov operators, the Stinespring representations methods used in the previous section to the deterministic case.
\\
Let $(\mathfrak M_1 ,\Omega_1) $  and $(\mathfrak M_2,\Omega_2) $ be standard von Neumann algebras in Hilbert spaces $\mathcal H_1$ and $\mathcal H_2$  respectively,  and $\Phi:(\mathfrak M_1,\Omega_1)\rightarrow(\mathfrak M_2,\Omega_2)$  a preserving Markov map with adjoint
 $\Phi^\sharp:(\mathfrak M_2,\Omega_2)\rightarrow (\mathfrak M_1,\Omega_1)$.
\\
We consider the Stinespring representations
$(\mathcal L_{\Phi^\sharp}, \sigma_{\Phi^\sharp}, V_{\Phi^\sharp})$  and $(\mathcal L_{\Phi^\sharp}, \tau_{\Phi^\sharp}, \Lambda_{\Phi^\sharp})$ of $\Phi^\sharp$.
\\
We assume that  there is anti-unitary operator $\widehat{J}:\mathcal L_{\Phi^\sharp} \rightarrow \mathcal L_{\Phi^\sharp}$ with the following property: 
\begin{equation}
\label{antiunij}
\widehat{J} V_{\Phi^\sharp}=V_{\Phi^\sharp} J_1 
\end{equation}
and we consider the von Neumann algebra $\mathfrak R $ of $\mathfrak B(\mathcal L_{\Phi^\sharp})$ generated by $ \sigma_{\Phi^\sharp}(\mathfrak M_2)$ and $\widehat{J} ^* \tau_{\Phi^\sharp}(\mathfrak M_1') \widehat{J}$.
\\
We remark that $\Omega_{\Phi^\sharp}\in\mathcal L_{\Phi^\sharp}$ is cyclic vector for $\mathfrak R $ since for each $A_2\in\mathfrak M_2$ and $Y_1\in\mathfrak M_1$ we obtain
$$\alpha(A_2) \widehat{J} ^* \tau_{\Phi^\sharp}(Y_1) \widehat{J} \Omega_{\Phi^\sharp}= A_2 \overline{\otimes}_{\Phi^\sharp} J_1 Y_1 \Omega_1 $$
We observe that
$$
 \mathfrak M_1\subset V_{\Phi^\sharp}^* \mathfrak R  V_{\Phi^\sharp} \quad \text{while }\quad \mathfrak M_2 \subset  \Lambda_{\Phi^\sharp}^* \mathfrak R \Lambda_{\Phi^\sharp}$$
We have the following proposition:
\begin{proposition}
Let  $\Phi:(\mathfrak M_1,\Omega_1)\rightarrow (\mathfrak M_2,\Omega_2)$ be a preserving Markov Operator and $\widehat{J}:\mathcal L_{\Phi^\sharp} \rightarrow \mathcal L_{\Phi^\sharp}$ which satisfies the relationships (\ref{antiunij}). If 
\begin{equation}
\label{cond1}
V_{\Phi^\sharp}^* \mathfrak R  V_{\Phi^\sharp} \subset \mathfrak M_1 \quad and \quad \Lambda_{\Phi^\sharp}^* \mathfrak R \Lambda_{\Phi^\sharp}\subset \mathfrak M_2
\end{equation}
then $(\mathfrak R, \Omega_{\Phi^\sharp})$ is standard in $\mathcal L_{\Phi^\sharp}$ and $\Phi$ is factorizable.
\end{proposition}
\begin{proof}
We prove that the vector $\Omega_{\Phi^\sharp}$ is separating for $\mathfrak R$.
\\
In fact, if $R\in\mathfrak R$ with $R\Omega_{\Phi^\sharp}=0$, then we can write that $\Lambda_{\Phi^\sharp} R^* R \Lambda_{\Phi^\sharp}\Omega_2=0$ and from relationships (\ref{cond1}) we have that $\Lambda_{\Phi^\sharp} R^* R \Lambda_{\Phi^\sharp}\in\mathfrak M_2$  with $(\mathfrak M_2, \Omega_2)$ standard in $\mathcal H_2$. It follows that  $\Lambda_{\Phi^\sharp} R^* R \Lambda_{\Phi^\sharp}=0$  hence $R \Lambda_{\Phi^\sharp}=0$. 
\\
In similar way we obtain that $R V_{\Phi^\sharp}=0$.
\\
For each $A_2\in\mathfrak M_2$ we have:
$$R \Lambda_{\Phi^\sharp} A_2 \Omega_2= R \sigma_{\Phi^\sharp}(A_2)\Omega_{\Phi^\sharp}=0$$
and repeating the reasoning for $R \sigma_{\Phi^\sharp}(A_2)\in \mathfrak R$ we obtain:
$$ R \sigma_{\Phi^\sharp}(A_2) \Lambda_{\Phi^\sharp}=0 \quad \text{and}  \quad R\sigma_{\Phi^\sharp}(A_2) V_{\Phi^\sharp}=0 $$
hence for each $A_1 \in \mathfrak M_1$ result 
$$R \sigma_{\Phi^\sharp}(A_2) V_{\Phi^\sharp} J_1 A_1 \Omega_1 = R \sigma_{\Phi^\sharp}(A_2) \beta(A_1) \Omega_{\Phi^\sharp}=0 $$ 
in other words $R A_2 \overline{\otimes}_{\Phi^\sharp } A_1 \Omega_1 =0 $ for all $A_2\in\mathfrak M_2$ and $A_1 \in\mathfrak M_1$.
\\
We consider  $\sigma_{\Phi^\sharp}:(\mathfrak M_2, \Omega_2) \rightarrow (\mathfrak R, \Omega_{\Phi^\sharp})$, we have 
$$\sigma_{\Phi^\sharp}^\sharp (T)= \Lambda_{\Phi^\sharp}^* T \Lambda_{\Phi^\sharp}\in\mathfrak M_2$$
for all  $T\in \mathfrak R$, since 
 \begin{eqnarray*}
\langle \Omega_2, A_2  \Lambda_{\Phi^\sharp}^* T \Lambda_{\Phi^\sharp} \Omega_2
 \rangle = \langle \Lambda_{\Phi^\sharp} A_2^* \Omega_2 ,T \Omega_{\Phi^\sharp} \rangle = 
\langle \Omega_{\Phi^\sharp}, \sigma_{\Phi^\sharp}(A_2) T \Omega_{\Phi^\sharp} \rangle 
\end{eqnarray*}
\\
We can define an another stochastic map $ \beta :(\mathfrak M_1, \Omega_1) \rightarrow (\mathfrak R, \Omega_{\Phi^\sharp})$ as
$$ \beta(A_1) = \widehat{J} ^* \tau_{\Phi^\sharp}(J_1 A_1 J_1)  \widehat{J}$$
for all $A_1\in\mathfrak M_1$, with 
$$\beta^\sharp (T)= V_{\Phi^\sharp}^* T V_{\Phi^\sharp}\in\mathfrak M_1$$
for all $T\in \mathfrak R $.
\\
Indeed
 \begin{eqnarray*}
\langle \Omega_{\Phi^\sharp}, \beta (A_1) T \Omega_{\Phi^\sharp} \rangle & = &
 \langle \Omega_{\Phi^\sharp}, \widehat{J}^* \tau_{\Phi^\sharp} (J_1 A_1 J_1) \widehat{J} T \Omega_{\Phi^\sharp} \rangle=
  \langle \widehat{J}^* \tau_{\Phi^\sharp} (J_1 A_1^* J_1)\Omega_{\Phi^\sharp},  T \Omega_{\Phi^\sharp} \rangle=
\\
& = &
\langle  \widehat{J}^* V_{\Phi^\sharp} J_1 A_1^* J_1 \Omega_1,  T \Omega_{\Phi^\sharp} \rangle=
\langle  V_{\Phi^\sharp}   A_1^* \Omega_1,  T \Omega_{\Phi^\sharp} \rangle=
\\
& = &
\langle  \Omega_1,   A_1 V_{\Phi^\sharp}^*T \Omega_{\Phi^\sharp} \rangle=
\langle \Omega_1, A_1\beta^\sharp (T) \Omega_1  \rangle
\end{eqnarray*}
Furthermore, we have: 
$$ \beta^\sharp(\sigma_{\Phi^\sharp}(A_2))=\Phi^\sharp(A_2) $$
for all $A_2\in\mathfrak M_2 $, hence $\Phi= \sigma_{\Phi^\sharp}^\sharp \circ \beta$.
\end{proof}
We observe that if $\widehat{J} \Lambda_{\Phi^\sharp}=\Lambda_{\Phi^\sharp} J_2 $ and $ \sigma_{\Phi^\sharp}(\mathfrak M_2) \subset \beta(\mathfrak M_1)'$ then the  relationships (\ref{cond1}) are satisfying, since $\mathfrak R$
is generated by set of elements
$$ \left\{ \sigma_{\Phi^\sharp}(A_2) \cdot  \beta(Y_1) : A_2\in\mathfrak M_2 \quad Y_1\in\mathfrak M_1   \right\}$$
and by relationships (\ref{st1}), (\ref{st2}) and (\ref{st5}) we have:
\begin{eqnarray*}
\Lambda_{\Phi^\sharp}^* \sigma_{\Phi^\sharp}(A_2) \widehat{J}^* \tau_{\Phi^\sharp}(Y_1) \widehat{J} \Lambda_{\Phi^\sharp} & = &
\Lambda_{\Phi^\sharp}^* \sigma_{\Phi^\sharp}(A_2) \Lambda_{\Phi^\sharp} \Lambda_{\Phi^\sharp}^*\widehat{J}^* \tau_{\Phi^\sharp}(Y_1) \widehat{J} \Lambda_{\Phi^\sharp}=
\\
 & = &
\Lambda_{\Phi^\sharp}^* \sigma_{\Phi^\sharp}( A_2)  \Lambda_{\Phi^\sharp} J_2 \Lambda_{\Phi^\sharp}^* \tau_{\Phi^\sharp}(Y_1) \Lambda_{\Phi^\sharp} J_2=
\\
 & = &
A_2 \Phi(J_1 Y_1 J_1) \in\mathfrak M_2
\end{eqnarray*}
for all $ A_2 \in\mathfrak M_2$ and $Y_1\in\mathfrak M_1'$.
\\
In similar way we have:
\begin{eqnarray*}
V_{\Phi^\sharp}^* \sigma_{\Phi^\sharp}(A_2) \widehat{J}^* \tau_{\Phi^\sharp}(Y_1) \widehat{J} V_{\Phi^\sharp} = \Phi^\sharp(A_2) J_1 Y_1 J_1 \in\mathfrak M_1
\end{eqnarray*}
We see some applications of the previous proposition.
\\

\textbf{Factorization in Abelian case}
\\
If $\Phi:(\mathfrak M_1,\Omega_1)\rightarrow (\mathfrak M_2,\Omega_2)$ is a preserving Markov operator between commutative probability spaces then is factorizable.
\\
Indeed, we consider the anti-unitary operator $W:\mathcal L_{\Phi^\sharp } \rightarrow \mathcal L_{\Phi'}$ defined in (\ref{antiuniw}) and the homomorphism $\beta(A_1)= W ^* \tau_{\Phi^\sharp}(J_1 A_1 J_1) W$ for all $A_1\in\mathfrak M_1$.
\\
In abelian case, because our von Neumann algebras are in standard form, we have that $\mathfrak M_i = \mathfrak M_i'$ for $i=1,2$ and $\Phi^\sharp = \Phi'$ with
$$W ^* \tau_{\Phi^\sharp}(\mathfrak M_1') W = \tau_{\Phi^\sharp}(\mathfrak M_1')= \tau_{\Phi^\sharp}(\mathfrak M_1)$$
since  $\tau_{\Phi^\sharp} = \tau_{\Phi'}$.
\\
We rematk that the anti-unitary $W$ is an involution \textit{i.e.} $W^2=1$.
\\
Hence, the von Neumann algebra $\mathfrak R $ is generated by algebra $\sigma_{\Phi^\sharp}(\mathfrak M_2)$ and  $\tau_{\Phi^\sharp}(\mathfrak M_1')$. From relationship (\ref{stdual}) and of the previous remark, we have that $\Omega_{\Phi^\sharp}$ is a cyclic and separable vector for $\mathfrak R$ and the pair $(\sigma_{\Phi^\sharp} , \beta)$ is a minimal factorization of $\Phi$.
\\

\textbf{Deterministic case}
\\
%
We consider again the deterministic  case, we proof that there is a anti-unitari operator $\widehat{J}$ which satisfies the relationship (\ref{antiunij}) and
$$\widehat{J} ^* \tau_{\Phi^\sharp}(\mathfrak M_1') \widehat{J} \subset\sigma_{\Phi^\sharp}(\mathfrak M_2)$$
in other words that $\mathfrak R  = \sigma_{\Phi^\sharp}(\mathfrak M_2) $.
\\
Because $\Omega_{\Phi^\sharp}$ is a cyclic vector for $\sigma_{\Phi^\sharp}(\mathfrak M_2)'$ we can consider the following anti-linear map
\begin{equation}
 \widehat{J} \ T' \Omega_{\Phi^\sharp} := J_2 \Lambda_{\Phi^\sharp}^* T' \Lambda_{\Phi^\sharp} J_2 \ \overline{\otimes}_{\Phi^\sharp} \Omega_1 \qquad T'\in\sigma_{\Phi^\sharp}(\mathfrak M_2)'
\end{equation}
We remark that 
$$\Lambda_{\Phi^\sharp}^* \sigma_\Phi(\mathfrak M_2)' \Lambda_{\Phi^\sharp}\subset (\Lambda_{\Phi^\sharp}^* \sigma_\Phi(\mathfrak M_2) \Lambda_{\Phi^\sharp})'=\mathfrak M_2 '$$
because $\Lambda_{\Phi^\sharp} \Lambda_{\Phi^\sharp}^*\in \sigma_{\Phi^\sharp}(\mathfrak M_2)'$.
\\
Furthermore, we have for any $A_2 \in\mathfrak M_2$ and $ h_1\in \mathcal H_1$ that 
$$\widehat{J}^* A_2  \overline{\otimes}_{\Phi^\sharp} h_1 = \Lambda_{\Phi^\sharp} J_2 A_2 U_\Phi h_1$$
since for any $T'\in\sigma_{\Phi^\sharp}(\mathfrak M_2)'$ and $A_2\in\mathfrak M_2$, $h_1\in \mathcal H_1$ we have
\begin{align*}  
\left\langle  \widehat{J}^* A_2 \overline{\otimes}_{\Phi^\sharp } h_1 , T' \Omega_{\Phi^ \sharp}\right\rangle & = 
\left\langle  \widehat{J}T' \Omega_{\Phi ^\sharp}  , A_2 \overline{\otimes}_{\Phi^ \sharp } h_1  \right\rangle =
\left\langle  J_2 \Lambda_{\Phi^\sharp}^* T' \Lambda_{\Phi^\sharp} J_2\overline{\otimes}_{\Phi^\sharp} \Omega_1 ,A_2 \overline{\otimes}_{\Phi^ \sharp } h_1  \right\rangle =
\\
& = 
\left\langle \Omega_1, \Phi^\sharp(J_2 \Lambda_{\Phi^\sharp}^* T'^* \Lambda_{\Phi^\sharp} J_2 A_2) h_1  \right\rangle =
\left\langle U^*_\Phi A_2^* J_2 \Lambda_{\Phi^\sharp}^* T' \Omega_{\Phi^\sharp}, h_1  \right\rangle =
\\
& =
\left\langle  \Lambda_{\Phi^\sharp} J_2 A_2 U_\Phi h_1, T' \Omega_{\Phi^\sharp} \right\rangle 
\end{align*} 
Moreover $ \widehat{J}^* \widehat{J}=\Lambda_{\Phi^\sharp} \Lambda_{\Phi^\sharp}^*$ and since $\Phi$ is a multiplicative map we have $\Lambda_{\Phi^\sharp} \Lambda_{\Phi^\sharp}^*=I$.
\\
Since $\Phi$ is a multiplicative map we have $ U_\Phi^*  A_2 U_\Phi= \Phi^\sharp(  A_2 )$ for all $A_2\in\mathfrak M_2$ and
\begin{align*}  
\left\langle  \widehat{J}^* A_2 \overline{\otimes}_{\Phi^\sharp } h_1 , \widehat{J}^* B_2 \overline{\otimes}_{\Phi^\sharp } k_1 \right\rangle & =
\left\langle \Lambda_{\Phi^\sharp} J_2 A_2 U_\Phi h_1 , \Lambda_{\Phi^\sharp} J_2 B_2 U_\Phi k_1\right\rangle =
\left\langle B_2 U_\Phi k_1 , A_2 U_\Phi h_1\right\rangle =
\\
& =
\left\langle k_1 , U_\Phi^* B_2^* A_2 U_\Phi h_1 \right\rangle =
\left\langle  k_1 , \Phi^\sharp( B_2^* A_2 ) h_1 \right\rangle =
\\
& =
\left\langle B_2 \overline{\otimes}_{\Phi^\sharp } K_1  , A_2 \overline{\otimes}_{\Phi^\sharp } h_1 \right\rangle 
\end{align*}
It follows that $ \widehat{J} \widehat{J}^*=I $.
\\
We observe that for any $R', T'\in\sigma_{\Phi^\sharp}(\mathfrak M_2)'$ and $A_1\in\mathfrak M_1$ we have
\begin{align*}  
\left\langle  R' \Omega_{\Phi^\sharp} \ , \ \widehat{J}^* \tau_{\Phi^\sharp} (J_1 A_1 J_1) \widehat{J} T' \Omega_{\Phi^\sharp } \right\rangle & =
\left\langle \tau_{\Phi^\sharp} (J_1 A_1 J_1) \widehat{J}^* T' \Omega_{\Phi^\sharp }  \ , \ \widehat{J} R' \Omega_{\Phi^\sharp} \right\rangle =
\\
& =
\left\langle J_2  \Lambda_{\Phi^\sharp }^* T' \Lambda_{\Phi^\sharp }J_2 \overline{\otimes}_{\Phi^\sharp } J_1 A_1 \Omega_1 \ , \ J_2  \Lambda_{\Phi^\sharp }^* R'\Lambda_{\Phi^\sharp }J_2  \overline{\otimes}_{\Phi^\sharp } \Omega_1\right\rangle =
\\
& =
\left\langle J_1 A_1 \Omega_1 \ , \  \Phi^\sharp (J_2  \Lambda_{\Phi^\sharp }^* T'^* \Lambda_{\Phi^\sharp }  \Lambda_{\Phi^\sharp }^* R' \Lambda_{\Phi^\sharp }J_2) \Omega_1 \right\rangle =
\\
& =
\left\langle  J_1 A_1 \Omega_1 \ ,U_\Phi^* J_2  \Lambda_{\Phi^\sharp }^* T'^* \Lambda_{\Phi^\sharp }  \Lambda_{\Phi^\sharp }^* R'\Omega_{\Phi^\sharp } \right\rangle =
\\
& =
\left\langle J_2  U_\Phi A_1\Omega_1 \ , J_2  \Lambda_{\Phi^\sharp }^* T'^* \Lambda_{\Phi^\sharp }  \Lambda_{\Phi^\sharp }^* R'\Omega_{\Phi^\sharp } \right\rangle =
\\
& =
\left\langle\Lambda_{\Phi^\sharp }^* T'^* \Lambda_{\Phi^\sharp }  \Lambda_{\Phi^\sharp }^* R'\Omega_{\Phi^\sharp } \ ,\Phi (A_1) \Omega_2  \right\rangle =
\\
& =
\left\langle \Lambda_{\Phi^\sharp }  \Lambda_{\Phi^\sharp }^* R'\Omega_{\Phi^\sharp } \ , T' \sigma_{\Phi^\sharp }(\Phi (A_1) \Omega_{\Phi^\sharp }) \right\rangle =
\\
& =
\left\langle R'\Omega_{\Phi^\sharp } \ , \Lambda_{\Phi^\sharp } \Lambda_{\Phi^\sharp }^* \sigma_{\Phi^\sharp }(\Phi (A_1) T' \Omega_{\Phi^\sharp }) \right\rangle
\end{align*}
Because $\Omega_{\Phi^\sharp }$ is cyclic for $\sigma_{\Phi^\sharp }(\mathfrak M_2)'$ we can write 
$$\widehat{J}^* \tau_{\Phi^\sharp} (J_1 A_1 J_1) \widehat{J}=\Lambda_{\Phi^\sharp } \Lambda_{\Phi^\sharp }^* \sigma_{\Phi^\sharp }(\Phi (A_1)=\sigma_{\Phi^\sharp }(\Phi (A_1)$$
thus
$$\widehat{J} ^* \tau_{\Phi^\sharp}(\mathfrak M_1') \widehat{J} \subset\sigma_{\Phi^\sharp}(\mathfrak M_2)$$
Moreover for any $ A_2 \overline{\otimes}_{\Phi^\sharp } h_1 \in\mathcal L_{\Phi^\sharp }$ we have
$$\Lambda_{\Phi^\sharp } ^* \widehat{J} ^*  A_2 \overline{\otimes}_{\Phi^\sharp } h_1 = J_2 A_2 U_\Phi h_1= J_2 \Lambda_{\Phi^\sharp } ^* A_2 \overline{\otimes}_{\Phi^\sharp } h_1$$
and
\begin{align*}
V_{\Phi^\sharp }^* \widehat{J} ^*  A_2 \overline{\otimes}_{\Phi^\sharp } h_1 & =  V_{\Phi^\sharp }^*  \Lambda_{\Phi^\sharp }  J_2 A_2 U_\Phi h_1 =
J_1 \Phi^\sharp ( A_2)  h_1 =
\\
&=
J_1 V_{\Phi^\sharp }^* A_2 \overline{\otimes}_{\Phi^\sharp } h_1
\end{align*}
We observe that the anti-unitary operator $\widehat{J}$ is an involution since:
\begin{align*}  
\left\langle  \widehat{J}^* A_2 \overline{\otimes}_{\Phi^\sharp } h_1 \ ,  \ \widehat{J} A_2 \overline{\otimes}_{\Phi^\sharp } h_1 \right\rangle & =
\left\langle  \Lambda J_2 A_2 U_\Phi h_1 \ , \ \widehat{J} A_2 \overline{\otimes}_{\Phi^\sharp } h_1  \right\rangle =
\\
& =
\left\langle  \Lambda J_2 A_2 U_\Phi h_1 \ , \ \widehat{J} \Lambda A_2 U_\Phi h_1 \right\rangle=
\\
& =
\left\langle  \Lambda J_2 A_2 U_\Phi h_1 \ , \ \Lambda  J_2 A_2 U_\Phi h_1 \right\rangle
\end{align*}
since $A_2 \overline{\otimes}_{\Phi^\sharp } h_1=\Lambda_{\Phi^\sharp} \Lambda_{\Phi^\sharp}^* A_2 \overline{\otimes}_{\Phi^\sharp } h_1=\Lambda_{\Phi^\sharp} A_2 U_\Phi h_1$.

\section{Factorization and generalized conditional expectation}
We recall briefly the notion of generalized conditional expectation of Accardi and Cecchini \cite{Accardi-Cecchini}.
\\ 
Let $(\mathfrak M, \varphi $ be a probability space and $\mathfrak R$ a von Neumann algebra with $i:\mathfrak M \rightarrow \mathfrak R$ a injective homomorphism between von Neumann algebras.
\\
We set with  the space of normal $\mathfrak S_\varphi$ the set:
$$ \mathfrak S_\varphi= \left\{ \omega\in\mathfrak R_* : \omega(i(a))=\varphi(a) \quad \text{for all} \ a\in\mathfrak M \right\} $$
where $\mathfrak R_*$ the predual of $\mathfrak R$.
\\
Let $(\mathcal H_s ,  \pi_s, J_s, \mathcal P_s)$  be a standard representation of algebra of von Neumann $\mathfrak R $ \cite{Haag}, it is widely know that 
there is a unique $\xi_\omega\in\mathcal P_s $ such that 
$$ \omega(X) = \left\langle  \xi_\omega \pi_s (X) \xi_\omega \right\rangle \quad \text{for all} \  X\in\mathfrak R $$
We define the following isometry $\nabla_\omega: \mathcal H_\varphi \rightarrow \mathcal H_s $:
$$ \nabla_\omega \pi_\varphi(a)\Omega_\varphi = \pi_s(i(a)) \xi_\omega \quad \text{for all} \ a\in\mathfrak M$$
and we obtain (see \cite{Accardi-Cecchini}) a unital completely positive map $\mathcal E_\omega : \mathfrak R \rightarrow \mathfrak M$ such that
$$\pi_\varphi(\mathcal E_\omega (X)) = J_\varphi \nabla_\omega^* \pi_s(X) J_s \nabla_\omega J_\varphi \qquad \text{for all} \  X\in\mathfrak R $$ 
Furthermore  $\mathcal E_\omega(i(a))=a$ for all $a\in\mathfrak M$ if, and only if $J_s \nabla_\omega= \nabla_\omega J_\varphi$.
\\

%
We consider again the preserving Markov Operator $\Phi:(\mathfrak M_1,\Omega_1)\rightarrow(\mathfrak M_2,\Omega_2)$ and the Stinespring representations $(\mathcal L_{\Phi^\sharp}, \sigma_{\Phi^\sharp}, V_{\Phi^\sharp})$  and $(\mathcal L_{\Phi^\sharp}, \tau_{\Phi^\sharp}, \Lambda_{\Phi^\sharp})$ related to adjoint map $\Phi^\sharp$.
\\
Let $\widehat J : \mathcal L_{\Phi^\sharp }  \rightarrow \mathcal L_{\Phi^\sharp } $ be an anti-unitary operator with the property (\ref{antiunij} ) and we set with $\mathfrak R$ the von Neumann algebra generated by $\sigma_{\Phi^\sharp}(\mathfrak M_2)$ and $\widehat{J} \tau_{\Phi^\sharp}(\mathfrak M_1') \widehat{J} $.
\\
Moreover, let $(\pi_s, \mathcal H_s , J_s, \mathcal P_s)$ be the standard representation of $\mathfrak R$, we define the following isometries $\nabla_i : \mathcal H_i \rightarrow \mathcal L_{\Phi^\sharp}$ as
\begin{align*}
 & \nabla_1 A_1 \Omega_1 = \pi_s(\beta(A_1)) \Omega_s \qquad  \quad A_1 \in \mathfrak M_1
\\
&
\nabla_2 A_2 \Omega_2 = \pi_s(\sigma_{\Phi^\sharp}(A_2)) \Omega_s \qquad  A_2 \in \mathfrak M_2 
\end{align*}
where $\beta(A_1)= \widehat{J} ^* \tau_{\Phi^\sharp}(J_1 A_1 J_1) \widehat{J}$ for all $A_1\in\mathfrak M_1$.
\\
We have a generalized conditional expectations $\mathcal E_i :\mathfrak R \rightarrow \mathfrak M_i$ with $i=1,2$ such that for each $R\in\mathfrak R$
\begin{equation} 
\label{GCE}
\mathcal E_i(R)= J_i \nabla_i^* J_s \pi_s(R) J_s \nabla_i J_i
\end{equation}
Furthermore  we have $\mathcal E_1(\beta(A_1))=A_1$ and $\mathcal E_1(\sigma_{\Phi^\sharp}(A_2))=A_2$ for all $A_i\in\mathfrak M_i$ with $i=1,2$  if, and only if 
\begin{equation}
\label{CCE}
J_s \nabla_i = \nabla_i J_i
\end{equation}
The vector $\Omega_{\Phi^\sharp}$ is cyclic for $\mathfrak R$ and we have the following:
\begin{proposition}
If the relationships (\ref{CCE}) is hold then $\Omega_{\Phi^\sharp}$ is a separating vector for $\mathfrak R$.
\\
Furthermore we have 
\begin{equation}
\left\langle  \Omega_{\Phi^\sharp} , R \sigma_{\Phi^\sharp} (A_2) \right\rangle =
\left\langle  \Omega_2 , \mathcal E_2 (R) A_2 \Omega_2 \right\rangle \qquad R\in\mathfrak R \quad A_2\in\mathfrak M_2
\end{equation}
and
\begin{equation}
\left\langle  \Omega_{\Phi^\sharp} , R \beta (A_1) \right\rangle =
\left\langle  \Omega_1 , \mathcal E_1 (R) A_1 \Omega_1 \right\rangle 
 \qquad R\in\mathfrak R \quad A_1\in\mathfrak M_1
\end{equation}
in other words $\mathcal E_1$ and $\mathcal E_2$ are adjoints maps of $\beta$ and $\alpha$ respectively.
\end{proposition}
\begin{proof}
The proof of separating property is similar the previous proposition . 
Indeed, let $R$ belongs to $\mathfrak R$ such that $R\Omega_{\Phi^\sharp}=0$. From (\ref{GCE}) we have that $\nabla_i^* R^* R \nabla_i \Omega_i =0$ for all $i=1,2$.
\\
It follows that $ R \nabla_i=0$ for all $i=1,2$. 
\\
For each $A_2\in\mathfrak M_2$ we obtain:
$$R \nabla_2 A_2 \Omega_2= R \sigma_{\Phi^\sharp}(A_2)\Omega_{\Phi^\sharp}=0$$
and repeating the argument for $R \sigma_{\Phi^\sharp}(A_2)\in \mathfrak R$ we obtain:
$$ R \sigma_{\Phi^\sharp}(A_2) \nabla_i=0 $$
and for each $A_1 \in \mathfrak M_1$ we have 
$$R \sigma_{\Phi^\sharp}(A_2) \nabla_1 A_1 \Omega_1= R \sigma_{\Phi^\sharp}(A_2) \beta(A_1) \Omega_{\Phi^\sharp}=0 $$ 
hence $R A_2 \overline{\otimes}_{\Phi^\sharp } A_1 \Omega_1 =0 $ for all $A_i\in\mathfrak M_i$ with $i=1,2$.
\\
We have for all $R\in\mathfrak R $ and $A_i\in\mathfrak M_i \quad i=1,2$
\begin{eqnarray*}
\left\langle  \Omega_2 , \mathcal E_2 (R) A_2 \Omega_2 \right\rangle =
\left\langle  \Omega_2 , \nabla_2^* R \nabla_2 A_2 \Omega_2 \right\rangle =
 \left\langle  \Omega_s , \pi_s(R \sigma_{\Phi^\sharp } (A_2) \Omega_s \right\rangle =
 \left\langle  \Omega_{\Phi^\sharp }  , R \sigma_{\Phi^\sharp } (A_2) \Omega_{\Phi^\sharp } \right\rangle
\end{eqnarray*}
while
\begin{eqnarray*}
\left\langle  \Omega_1 , \mathcal E_1 (R) A_1 \Omega_1 \right\rangle =
\left\langle  \Omega_1 , \nabla_1^* R \nabla_1 A_1 \Omega_1 \right\rangle =
 \left\langle  \Omega_s , \pi_s (R \beta (A_1) \Omega_s \right\rangle =
 \left\langle  \Omega_{\Phi^\sharp }  , R \beta(A_1) \Omega_{\Phi^\sharp } \right\rangle
\end{eqnarray*}
\end{proof}
We can give the following proposition:
\begin{corollary}
If the relationships (\ref{CCE}) is hold then $\Phi$ is a factorizable map.
\end{corollary}
\begin{proof}
We have that $\Phi^\sharp (A_1)= \beta^\sharp (\alpha(A_1))$ for all $A_1 \in \mathfrak M_1$.
\\
Indeed
\begin{eqnarray*}
\left\langle  A_1 \Omega_1 , \beta^\sharp (\alpha(A_2)) \Omega_1 \right\rangle & = & 
\left\langle  A_1 \Omega_1 , \nabla_1^* \pi_s(\sigma_{\Phi^\sharp } (A_2)) \nabla_1 \Omega_1 \right\rangle =
 \left\langle  \pi_s (\beta(A_1)) \Omega_s , \pi_s (\sigma_{\Phi^\sharp } (A_2)) \Omega_s \right\rangle =
\\
& = &
 \left\langle  \beta(A_1) \Omega_{\Phi^\sharp } , \sigma_{\Phi^\sharp } (A_2) \Omega_{\Phi^\sharp } \right\rangle =
 \left\langle  1\overline{\otimes}_{\Phi^\sharp } A_1 \Omega_1,  A_2 \overline{\otimes}_{\Phi^\sharp } \Omega_1 \right\rangle =
\\
& = &
 \left\langle  A_1 \Omega_1, \Phi^\sharp  (A_2) \Omega_1 \right\rangle
\end{eqnarray*}
\end{proof}
We have an isometry $\Xi:\mathcal L_{\Phi^\sharp } \rightarrow \mathcal H_s$ such that
$$\Xi A_2 \overline{\otimes}_{\Phi^\sharp } A_1\Omega_1=\pi_s(\sigma_{\Phi^\sharp } (A_2)\beta(A_1))\xi_\omega$$
for all $A_i\in\mathfrak M_i$ with $i=1,2$, where
$$\left\langle \Omega_{\Phi^\sharp } , X \Omega_{\Phi^\sharp} \right\rangle=
\left\langle \xi_\omega , \pi_s (X)  \xi_\omega \right\rangle  \quad \text{for all} \quad  X\in\mathfrak R $$ 
Moreover 
\begin{align*}
\Xi   V_{\Phi^\sharp } = \nabla_1 \qquad and \qquad \Xi   \Lambda_{\Phi^\sharp } = \nabla_2
\end{align*}
and if the  (\ref{CCE}) is hold, we can write a relationship between the anti-unitary $\widehat J$ and the modular coniugation $J_s$: 
$$\widehat J V_{\Phi^\sharp }= \Xi^* J_s \Xi V_{\Phi^\sharp }$$
%
We observe that for each $R\in\mathfrak R$ we have:
\begin{align*}
 V_{\Phi^\sharp}^* R V_{\Phi^\sharp}= \nabla_1^* \pi_s (R) \nabla_1 \qquad \text{and} \qquad \Lambda_{\Phi^\sharp}^* R \Lambda_{\Phi^\sharp} = \nabla_2^* \pi_s (R) \nabla_2
\end{align*}
Indeed for each $A_i,B_i\in\mathfrak M_i$ with $i=1,2$ we can write:
\begin{align*}
\left\langle B_1 \Omega_1 , \nabla_1^* \pi_s (R)  \nabla_1 A_1 \Omega_1 \right\rangle & =
\left\langle \xi_\omega, \pi_s(\beta(B_1^*)) \pi_s(R) \pi_s(\beta(A_1)) \xi_\omega\right\rangle=
\\
& =
\left\langle \Omega_{\Phi^\sharp}, \beta(B_1^*) R\beta(A_1) \Omega_{\Phi^\sharp} \right\rangle=
\\
&=
\left\langle  \widehat J V_{\Phi^\sharp} J_1 B_1 \Omega_1 ,  R \widehat J V_{\Phi^\sharp} J_1 A_1 \Omega_1 \right\rangle=
\\
&=
\left\langle  V_{\Phi^\sharp} B_1 \Omega_1 ,  R V_{\Phi^\sharp} A_1 \Omega_1 \right\rangle=
\\
&=
\left\langle  B_1 \Omega_1 ,  V_{\Phi^\sharp}^*  R V_{\Phi^\sharp} A_1 \Omega_1 \right\rangle
\end{align*}
while
\begin{align*}
\left\langle B_2 \Omega_2 , \nabla_2^* \pi_s (R)  \nabla_2 A_2 \Omega_2 \right\rangle & =
\left\langle \xi_\omega, \pi_s(\sigma_{\Phi^\sharp}(B_2^*)) \pi_s(R) \pi_s(\sigma_{\Phi^\sharp}(A_2)) \xi_\omega\right\rangle=
\\
& =
\left\langle \Lambda_{\Phi^\sharp} B_2 \Omega_2,  R \Lambda_{\Phi^\sharp} A_2 \Omega_2 \right\rangle=
\\
&=
\left\langle   B_2 \Omega_2 ,  \Lambda_{\Phi^\sharp} ^* R \Lambda_{\Phi^\sharp} A_2 \Omega_2 \right\rangle
\end{align*}
We can give a simple remark:
\begin{remark}
We have that $J_s \nabla_i = \nabla_i J_i$ with $i=1,2$ if,  and only if
\begin{align*}
 \mathcal E_1(R)= V^* R V  \qquad \text{and} \qquad \mathcal E_2 (R)= \Lambda^* R \Lambda 
\end{align*}
for all $R\in \mathfrak R$.
\end{remark}

\section{Conclusion}
In this paper we have given a simple method to determine when a preserving Markov operator is factorizable. It is based on the appropriate selection of an anti-unitary operator $\widehat{J}$ on Hilbert space of Stinespring representation $\mathcal L_{\Phi^\sharp }$. 
\\
A useful tool to establish the anti-unitary operator $\widehat{J}$ be found in  \cite{Maje1} and  \cite{Maje2} since there is a strong connection between coniugation operator and antilinear Jordan map of von Neumann algebra. Indeed, for each antilinear Jordan map $\gamma:\sigma_{\Phi^\sharp }(\mathfrak M_2)' \rightarrow \sigma_{\Phi^\sharp }(\mathfrak M_2)'$ \textit{i.e.}
\begin{itemize}
\item [1.]  $\gamma$ is an antilinear bijection
\item [2.]  $\gamma \circ \gamma = i$  where $i$ is identy map
\item [3.]  $\gamma(T^*)=\gamma(A)^*$  for all $T\in\sigma_{\Phi^\sharp }(\mathfrak M_2)'$
\item [4.]  $\gamma(\left\{ T,S \right\})=\left\{ \gamma(T), \gamma(S)\right\} $ for all $T,S\in\sigma_{\Phi^\sharp }(\mathfrak M_2)'$
\end{itemize}
we can define a coniugation $\widehat{J}$ on Hilbert space $\mathcal L_{\Phi^\sharp }$  \cite{Maje1} as  
$\widehat{J} T \Omega_{\Phi^\sharp }= \gamma(T) \Omega_{\Phi^\sharp }$ 
 for all $T\in\sigma_{\Phi^\sharp }(\mathfrak M_2)'$ .
\\
We remark that for the relationship (\ref{antiunij}), this Jordan map  must necessarily satisfy the following property:
$$\gamma(\tau_{\Phi^\sharp }(y))\Omega_{\Phi^\sharp }=V_{\Phi^\sharp } J_1 y\Omega_1$$
for all $y\in\mathfrak M_1 '$.

\end{document}